\documentclass[11pt]{amsart}
\usepackage{amsmath,amsthm,amsfonts,amssymb,latexsym,epsfig}

\newtheorem{theorem}{Theorem}[section]

\newtheorem{lemma}{Lemma}[section]

\newtheorem{cor}{Corollary}[section]

\numberwithin{equation}{section}

\theoremstyle{definition}

\theoremstyle{remark}

\begin{document}
\title{On an inequality suggested by Littlewood}
\author{Peng Gao}
\address{Division of Mathematical Sciences, School of Physical and Mathematical Sciences,
Nanyang Technological University, 637371 Singapore}
\email{penggao@ntu.edu.sg}


\begin{abstract}
  We study an inequality suggested by Littlewood, our result
  refines a result of Bennett.
\end{abstract}

\maketitle
\section{Introduction}
\label{sec 1} \setcounter{equation}{0}

  In connection with work on the general theory of orthogonal
  series, Littlewood \cite{Littlewood} raised some problems concerning elementary inequalities for infinite
  series. One of them asks to decide whether an absolute constant $K$ exists such that for any non-negative sequence $(a_n)$
with $A_n=\sum^n_{k=1}a_k$,
\begin{equation}
\label{1}
   \sum^{\infty}_{n=1}a_nA^2_n\left ( \sum^{\infty}_{k=n}a_k^{3/2} \right )^2 \leq K
   \sum^{\infty}_{n=1}a^2_nA^4_n.
\end{equation}

   The above problem was solved by Bennett \cite{B1}, who proved the
   following more general result:
\begin{theorem}[{\cite[Theorem 4]{B1}}]
\label{thm1}
   Let $p \geq 1, q>0, r >0$ satisfying $(p(q+r)-q)/p \geq 1$ be fixed. Let
   $K(p,q,r)$ be the best possible constant
   such that for any non-negative sequence $(a_n)$ with $A_n=\sum^n_{k=1}a_k$,
\begin{equation}
\label{2}
   \sum^{\infty}_{n=1}a^p_nA^q_n\left ( \sum^{\infty}_{k=n}a_k^{1+p/q} \right )^r \leq
  K(p,q,r)
   \sum^{\infty}_{n=1}\left (a^p_nA^q_n \right )^{1+r/q}.
\end{equation}
   Then
\begin{equation*}
    K(p,q,r) \leq \left ( \frac {p(q+r)-q}{p} \right )^r.
\end{equation*}
\end{theorem}

   The special case $p=1, q=r=2$ in \eqref{2} leads to inequality
   \eqref{1} with $K=4$ and Theorem \ref{thm1} implies that $K(p,q,r)$ is finite for any $p \geq 1, q>0, r>0$ satisfying $(p(q+r)-q)/p \geq 1$, a fact we
   shall use implicitly throughout this paper. We note that Bennett only
proved Theorem \ref{thm1} for $p, q, r \geq 1$ but as was pointed
out in \cite{Alzer}, Bennett's proof
   actually works for the $p, q, r$'s satisfying the condition in Theorem \ref{thm1}. Another proof of inequality \eqref{2}
   for the special case $r=q$ was provided by Bennett in \cite{Be3}
   and a close look at the proof there shows that it in fact can be used to establish Theorem
   \ref{thm1}.

   On setting $p = 2$ and $q=r= 1$ in \eqref{2}, and interchanging the order of summation on the left-hand side of \eqref{2},
   we deduce the following
\begin{equation}
\label{3}
   \sum^{\infty}_{n=1}a^3_n\sum^{n}_{k=1}a_k^{2}A_k \leq
   \frac {3}{2}
   \sum^{\infty}_{n=1}a^4_nA^2_n.
\end{equation}

   The constant in \eqref{3} was improved to be $2^{1/3}$ in \cite{G} and the following more general result was given in
   \cite{CTZ}:
\begin{theorem}[{\cite[Theorem 2]{CTZ}}]
   Let $p, q \geq 1, r>0$ be fixed satisfying $r(p-1) \leq 2(q-1)$. Set
\begin{equation*}
  \alpha=\frac {(p-1)(q+r)+p^2+1}{p+1}, \ \beta=\frac
  {2q+2r+p-1}{p+1},  \ \delta=\frac {q+r-1}{p+q+r}.
\end{equation*}
   Then for any non-negative sequence $(a_n)$ with $A_n=\sum^n_{k=1}a_k$,
\begin{equation}
\label{4}
    \sum^{\infty}_{n=1}a^p_n\sum^n_{k=1}a^q_kA^r_k \leq
    2^{\delta}\sum^{\infty}_{n=1}a^{\alpha}_{n}A^{\beta}_n.
\end{equation}
\end{theorem}
   Note that inequality \eqref{3} with constant
   $2^{1/3}$ corresponds to the case $p=3, q=2, r=1$ in \eqref{4}.
    In \cite{Z}, an even better constant was obtained but the proof there is
    incorrect. In \cite{Alzer}, \cite{CTZ} and \cite{Z}, results
    were also obtained concerning inequality \eqref{2} under the
    extra assumption that the sequence $(a_n)$ is non-decreasing.

  The exact value of $K(p,q,r)$ is not known in general. But note that $K(1,q,1)=1$ as it follows immediately from Theorem \ref{thm1}
  that $K(1,q,1) \leq 1$ while on the other hand on setting
  $a_1=1, a_n=0, n \geq 2$ in \eqref{2} that $K(1,q,1) \geq 1$.
  Therefore we may restrict our attention on \eqref{2} for $p, r$'s not
  both being $1$. It's our goal in this paper to improve the result in Theorem
\ref{thm1} in the following
\begin{theorem}
\label{thm2} Let $p \geq 1,  q>0 , r \geq 1$ be fixed with $p,r$
not both being $1$. Under the same notions of Theorem \ref{thm1},
inequality \eqref{2} holds when $q+r-q/p \geq 2$ with
\begin{equation*}
   K(p,q,r) \leq  K \left ( p\left (1+\frac {r-1}{q} \right ), q+r-1,
1 \right )\left ( \frac {p(q+r)-q}{p} \right )^{r-1}.
\end{equation*}
   When  $1 \leq q+r-q/p \leq 2$, inequality \eqref{2} holds with
\begin{equation*}
   K(p,q,r) \leq \left (K \left ( p\left (1+\frac {r-1}{q} \right ), q+r-1,
1 \right ) \right )^{r-\frac {p(r-1)}{pq+p(r-1)-q}} \left ( \frac
{p(q+r)-q}{p} \right )^{\frac {p(r-1)}{pq+p(r-1)-q}}.
\end{equation*}
   Moreover, for any $p \geq 1, q >0$,
\begin{equation}
\label{7}
   K(p,q,1) \leq \min_{\delta} \left ( \frac {p(q+1)-q}{p}, \ C(p,q, \delta)   \right
   ),
\end{equation}
   where
\begin{align}
\label{8}
   C(p,q, \delta)=\left ( \delta \left (1+\frac {p}{q(p-1)}\right )\left (1+\frac
   {1}{1/(p-1)+\delta(1+p/(q(p-1)))-1} \right) \right )^{\delta},
\end{align}
   and the minimum in \eqref{7} is taken over the $\delta$'s satisfying
\begin{align}
\label{9}
   \frac {q(p-1)}{p(q+1)-q} \leq \delta \leq 1.
\end{align}
\end{theorem}

    On considering the values of $C(p, q, \delta)$ for $\delta=1$
    and $\delta=q(p-1)/(p(q+1)-q)$, we readily deduce from Theorem
    \ref{thm2} the following
\begin{cor}
\label{cor0}
   Let $p \geq 1,  q>0$ be fixed. Let $K(p,q,r)$ be the
best possible constant
   such that inequality \eqref{2} holds for any non-negative sequence $(a_n)$. Then
\begin{equation}
\label{5}
   K(p,q,1) \leq \min \left ( \frac {p(q+1)-q}{p}, \ p^{\frac {(p-1)q}{(p-1)q+p}} \ , \ \left (1+\frac {(p-1)q}{p+q} \right )\left (1+\frac {p}{q(p-1)}\right )   \right ).
\end{equation}
\end{cor}

    We note that Theorem \ref{thm2} together with Lemma
   \ref{reductionlem1} below shows that a bound for $K(p,q,r)$ with $p \geq 1, q>0, r >0$ satisfying $(p(q+r)-q)/p \geq
   1$ can be obtained by a bound of $K(p(1+(r-1)/q), q+r-1,1)$ and
   as \eqref{5} implies that $K(p (1+(r-1)/q), q+r-1,1) \leq
    (p(q+r)-q)/p$, it is easy to see that the assertion of
    Theorem \ref{thm1} follows from the assertions of
    Theorem \ref{thm2} and Lemma \ref{reductionlem1}.

  We point out here that among the three expressions on the
right-hand side of \eqref{5}, each one is likely to be the
minimum.
    For example, the middle one becomes the minimum when $p=2, q=1$
    while it's easy to see that the last one becomes the minimum for $p=q$ large enough and the first one becomes the minimum
     when $q$ is being fixed and $p \rightarrow \infty$. Moreover, it can happen that
    the minimum value in \eqref{7} occurs at a $\delta$ other than $q(p-1)/(p(q+1)-q),1$.
    For example, when $p=q=6$, the bound \eqref{5} gives $K(6,6,1) \leq 21/5$ while one checks easily
    that $C(6,6, 1.15/1.2)<21/5$. We shall not worry about determining the precise minimum
    of \eqref{7} in this paper.

   We note that the special case $p=1, q=r=2$ of Theorem \ref{thm2}
   leads to the following improvement on Bennet's result on the constant
   $K$ of inequality \eqref{1}:
\begin{cor}
\label{cor1}
   Inequality \eqref{1} holds with $K=\sqrt{6}$.
\end{cor}

\section{A few Lemmas}
\label{sec 2} \setcounter{equation}{0}
\begin{lemma}
\label{lem1} Let $d \geq c >1$ and $(\lambda_n)$ be a non-negative
sequence with $\lambda_1>0$. Let $
\Lambda_n=\sum^n_{k=1}\lambda_k$. Then for all non-negative
sequences $(x_n)$,
\begin{equation*}
  \sum^{\infty}_{n=1}\lambda_n\Lambda^{-c}_n\left ( \sum^n_{k=1}\lambda_kx_k\right )^d \leq
  \left ( \frac {d}{c-1}\right
  )^d\sum^{\infty}_{n=1}\lambda_n\Lambda^{d-c}_nx^d_n.
\end{equation*}
    The constant is best possible.
\end{lemma}
   The above lemma is the well-known Copson's inequality \cite[Theorem 1.1]{Copson}, see also
   Corollary 3 to Theorem 2 of \cite{B1}.

\begin{lemma}
\label{lem2} Let $p < 0 $. For any non-negative sequence $(a_n)$
with $a_1>0$ and $A_n=\sum^n_{k=1}a_k$, we have for any $n \geq
1$,
\begin{equation}
\label{2.2}
  \sum^{\infty}_{k=n}a_kA^{p-1}_k \leq \left (1-\frac {1}{p} \right
  )A^p_n.
\end{equation}
\end{lemma}
\begin{proof}
    We start with the inequality $x^p - px + p - 1 \geq 0$. By setting
$x = A_{k-1}/A_k$ for $k \geq 2$, we obtain
\begin{equation*}
  A^p_{k-1}-pA_{k-1}A^{p-1}_k+(p-1)A^p_k \geq 0.
\end{equation*}
   Replacing $A_{k-1}$ in the middle term of the left-hand side
   expression above by $A_k-a_k$ and simplifying, we obtain
\begin{equation*}
  A^p_{k-1}-A^p_k \geq -pa_kA^{p-1}_k.
\end{equation*}
   Upon summing, we obtain
\begin{equation*}
  \sum^{\infty}_{k=n+1}a_kA^{p-1}_k \leq -\frac {1}{p}A^p_n.
\end{equation*}
   Inequality \eqref{2.2} follows from above upon noting that
   $a_nA^{p-1}_n \leq A^p_n$.
\end{proof}

\begin{lemma}
\label{reductionlem}
     Let $p \geq 1, q>0, r \geq 1$ be fixed with $p,r$ not both
being $1$. Under the same notions of Theorem \ref{thm1}, we have
\begin{align*}
    & K(p,q,r)  \\
\leq& \left (K \left ( p\left (1+\frac {r-1}{q} \right ), q+r-1, 1
\right ) \right )^{\frac {p-1}{p+p(r-1)/q-1}} \left ( K \left (1,
\frac {q}{p}, \frac {p(q+r)-q}{p} \right ) \right )^{\frac
{p(r-1)}{pq+p(r-1)-q}}.  \nonumber
\end{align*}
\end{lemma}
\begin{proof}
   As it is easy to check the assertion of the lemma holds when $p=1$ or $r=1$, we may assume $p>1, r>1$ here. We set
\begin{align*}
   \alpha=\frac {p-1}{p+p(r-1)/q-1}, \ \beta=\alpha \left(1+\frac
   {r-1}{q} \right ), \ b_n=a_nA^{q/p}_n, \ c_n=\sum^{\infty}_{k=n}a^{1+p/q}_k.
\end{align*}
   Note that we have $0<\alpha<1$ as $p>1, r>1$ here. By
   H\"older's inequality, we have
\begin{align}
\label{3.5}
   & \sum^{\infty}_{n=1}a^p_nA^q_n\left ( \sum^{\infty}_{k=n}a_k^{1+p/q} \right
    )^{r}=\sum^{\infty}_{n=1}b^p_nc^r_n = \sum^{\infty}_{n=1}b^{p\beta}_nc^{\alpha}_n \cdot
b^{p(1-\beta)}_nc^{r-\alpha}_n  \\
\leq & \left ( \sum^{\infty}_{n=1}b^{p\beta/\alpha}_nc_n
\right)^{\alpha} \cdot \left (
\sum^{\infty}_{n=1}b^{p(1-\beta)/(1-\alpha)}_nc^{(r-\alpha)/(1-\alpha)}_n
\right)^{1-\alpha}  \nonumber \\
= &  \left ( \sum^{\infty}_{n=1}b^{p(1+(r-1)/q)}_nc_n
 \right)^{\frac
{p-1}{p+p(r-1)/q-1}} \cdot \left (
\sum^{\infty}_{n=1}b_nc^{(p(q+r)-q)/p}_n \right)^{\frac
{p(r-1)}{pq+p(r-1)-q}}. \nonumber
\end{align}
  The assertion of the lemma now follows on applying inequality
  \eqref{2} to both factors of the last expression above.
\end{proof}

\begin{lemma}
\label{reductionlem1}
     Let $p \geq 1, q>0, 0< r \leq 1$ be fixed satisfying $(p(q+r)-q)/p \geq 1$. Under the same notions
of Theorem \ref{thm1}, we have
\begin{align}
\label{prebound1}
    K(p,q,r) \leq  \left (K \left ( p\left (1+\frac {r-1}{q} \right ), q+r-1, 1
\right ) \right )^{r}.
\end{align}
\end{lemma}
\begin{proof}
   We may assume $0<r<1$ here. We set
\begin{align*}
   \alpha=1-r, \ \beta=\alpha \left(1+\frac
   {r}{q} \right ), \ b_n=a_nA^{q/p}_n, \ c_n=\sum^{\infty}_{k=n}a^{1+p/q}_k.
\end{align*}
   Note that we have $0<\alpha<1$. By
   H\"older's inequality, we have
\begin{align*}
   & \sum^{\infty}_{n=1}a^p_nA^q_n\left ( \sum^{\infty}_{k=n}a_k^{1+p/q} \right
    )^{r}=\sum^{\infty}_{n=1}b^p_nc^r_n = \sum^{\infty}_{n=1}b^{p\beta}_n \cdot b^{p(1-\beta)}_n c^{r}_n \ \\
\leq & \left ( \sum^{\infty}_{n=1}b^{p\beta/\alpha}_n
\right)^{\alpha}  \left (
\sum^{\infty}_{n=1}b^{p(1-\beta)/(1-\alpha)}_nc^{r/(1-\alpha)}_n
\right)^{1-\alpha}  \nonumber \\
= &  \left ( \sum^{\infty}_{n=1}b^{p(1+r/q)}_n
  \right)^{1-r} \left (
\sum^{\infty}_{n=1}b^{p(1+(r-1)/q)}_nc_n \right)^{r}. \nonumber
\end{align*}
  The assertion of the lemma now follows on applying inequality
  \eqref{2} to the second factor of the last expression above.
\end{proof}
\section{Proof of Theorem \ref{thm2}}
\label{sec 3} \setcounter{equation}{0}

    We obtain the proof of Theorem \ref{thm2} via the following
    two lemmas:
\begin{lemma}
\label{basiclem}
   Let $p \geq 1, q>0$ be fixed.  Under the same notions of Theorem \ref{thm1}, inequality \eqref{2} holds when $r=1$
   with $K(p,q,1)$ bounded by the right-hand side expression of \eqref{7}.
\end{lemma}
\begin{proof}
    We may assume that only finitely many $a_n$'s are positive, say $a_n = 0$ whenever $n >
    N$. We may also assume $a_1>0$. As the case $p=1$ of the lemma is already contained in Theorem \ref{thm1}, we may further assume
    $p>1$ throughout the proof. Moreover, even though the assertion
    that $K(p,q,1) \leq (p(q+1)-q)/p$ is already given in Theorem \ref{thm1}, we
include a new proof here.

    We recast the left-hand side expression of \eqref{2}
    as
\begin{align}
\label{3.1}
   & \sum^{N}_{n=1}a^p_nA^q_n \sum^{N}_{k=n}a_k^{1+p/q}
   =\sum^{N}_{n=1}a_n^{1+p/q}\sum^{n}_{k=1}a^p_kA^q_k  \\
   =& \sum^{N}_{n=1}\left (a^p_nA^q_n \right )^{\theta(1+1/q)} \cdot a_n^{1+p/q}
   \left (a^p_nA^q_n \right )^{-\theta(1+1/q)}
   \sum^{n}_{k=1}a^p_kA^q_k  \nonumber \\
   \leq & \left (\sum^{N}_{n=1}\left (a^p_nA^q_n \right )^{1+1/q} \right
   )^{\theta} \left ( \sum^{N}_{n=1} a_n^{(1+p/q)/(1-\theta)}
   \left (a^p_nA^q_n \right )^{-\theta(1+1/q)/(1-\theta)}
   \left (\sum^{n}_{k=1}a^p_kA^q_k \right )^{1/(1-\theta)} \right
   )^{1-\theta} \nonumber \\
   = & \left (\sum^{N}_{n=1}\left (a^p_nA^q_n \right )^{1+1/q} \right
   )^{\theta} \left ( \sum^{N}_{n=1} a_nA^{-p(q+1)/(q(p-1))}_n
   \left (\sum^{n}_{k=1}a^p_kA^q_k \right )^{(p(q+1)-q)/(q(p-1))} \right
   )^{1-\theta} ,  \nonumber
\end{align}
   where we set
\begin{equation*}
  \theta=\frac {p}{p(q+1)-q},
\end{equation*}
   so that $0<\theta<1$ and the inequality in \eqref{3.1} follows from an
   application of H\"older's inequality.

   Thus, in order to prove Theorem \ref{thm2}, it suffices to show that
\begin{align}
\label{auxbound}
   \sum^{N}_{n=1} a_nA^{-p(q+1)/(q(p-1))}_n
   \left (\sum^{n}_{k=1}a^p_kA^q_k \right )^{(p(q+1)-q)/(q(p-1))}
   \leq K_1(p,q)
   \sum^{N}_{n=1}\left (a^p_nA^q_n \right )^{1+1/q},
\end{align}
   where
\begin{align*}
    K_1(p,q) = \min_{\delta} \left ( \left (\frac {p(q+1)-q}{p} \right )^{1/(1-\theta)}, \ C \left(p,q, \delta \right )^{1/(1-\theta)}  \right
),
\end{align*}
    where $C(p,q,\delta)$ is defined as in \eqref{8} and the
minimum is taken over the $\delta$'s satisfying \eqref{9}.

   Note first we have
\begin{align*}
   & \sum^{N}_{n=1} a_nA^{-p(q+1)/(q(p-1))}_n
   \left (\sum^{n}_{k=1}a^p_kA^q_k \right )^{(p(q+1)-q)/(q(p-1))}
   \nonumber \\
   \leq & \sum^{N}_{n=1} a_nA^{-p(q+1)/(q(p-1))}_n
   \left (\sum^{n}_{k=1}a^p_kA^{q-q/(p(q+1)-q)}_kA^{q/(p(q+1)-q)}_n \right
   )^{(p(q+1)-q)/(q(p-1))} \nonumber \\
   =& \sum^{N}_{n=1} a_n
   \left (\sum^{n}_{k=1}\frac {a_k}{A_n}\left (a^{p-1}_kA^{q-q/(p(q+1)-q)}_k \right ) \right
   )^{(p(q+1)-q)/(q(p-1))} \nonumber \\
   \leq &  \left (\frac {p(q+1)-q}{p} \right )^{(p(q+1)-q)/(q(p-1))}\sum^{N}_{n=1}\left (a^p_nA^q_n \right
   )^{1+1/q}, \nonumber
\end{align*}
   where the last inequality above follows from Lemma \ref{lem1}
   by setting $d=c=(p(q+1)-q)/(q(p-1)), \lambda_n=a_n,
   x_n=a^{p-1}_nA^{q-q/(p(q+1)-q)}_n$ there. This establishes \eqref{auxbound} with
\begin{align*}
   K_1(p,q)=\left (\frac {p(q+1)-q}{p} \right )^{1/(1-\theta)}.
\end{align*}

  Next, we use the idea in \cite{G} (see also \cite{CTZ}) to see
  that for any $0 < \delta \leq 1$,
   \begin{align}
\label{4.1}
   & \sum^{N}_{n=1} a_nA^{-p(q+1)/(q(p-1))}_n
   \left (\sum^{n}_{k=1}a^p_kA^q_k \right )^{(p(q+1)-q)/(q(p-1))}
   \\
   = & \sum^{N}_{n=1} a_nA^{-1/(p-1)}_n
   \left (\sum^{n}_{k=1}\frac {a_k}{A_n}a^{p-1}_kA^q_k \right
   )^{(p(q+1)-q)/(q(p-1))} \nonumber \\
   = & \sum^{N}_{n=1} a_nA^{-1/(p-1)}_n
   \left (\sum^{n}_{k=1}\frac {a_k}{A_n}\left (a^{(p-1)/\delta}_kA^{q/\delta}_k \right )^{\delta} \right
   )^{(p(q+1)-q)/(q(p-1))} \nonumber \\
    \leq & \sum^{N}_{n=1} a_nA^{-1/(p-1)}_n
   \left (\sum^{n}_{k=1}\frac {a_k}{A_n}a^{(p-1)/\delta}_kA^{q/\delta}_k \right
   )^{\delta(p(q+1)-q)/(q(p-1))}.  \nonumber
\end{align}
   We now further require that
\begin{align*}
   \frac {q(p-1)}{p(q+1)-q} < \delta \leq 1,
\end{align*}
   then on setting for $1 \leq n \leq N$,
\begin{align*}
   S_n=\sum^N_{k=n} a_kA^{-1/(p-1)-\delta(p(q+1)-q)/(q(p-1))}_k, \ T_n=\left (\sum^{n}_{k=1}a^{1+(p-1)/\delta}_kA^{q/\delta}_k \right
   )^{\delta(p(q+1)-q)/(q(p-1))},
\end{align*}
   we have by partial summation, with $T_0=0$,
\begin{align}
\label{4.2}
   &  \sum^{N}_{n=1} a_nA^{-1/(p-1)}_n
   \left (\sum^{n}_{k=1}\frac {a_k}{A_n}a^{(p-1)/\delta}_kA^{q/\delta}_k \right
   )^{\delta(p(q+1)-q)/(q(p-1))}
   \\
   = & \sum^{N}_{n=1} a_nA^{-1/(p-1)-\delta(p(q+1)-q)/(q(p-1))}_n
   \left (\sum^{n}_{k=1}a^{1+(p-1)/\delta}_kA^{q/\delta}_k \right
   )^{\delta(p(q+1)-q)/(q(p-1))} \nonumber \\
   = & \sum^{N}_{n=1}S_n(T_n-T_{n-1}) \nonumber \\
   \leq & \delta \left (1+\frac {p}{q(p-1)}\right ) \sum^{N}_{n=1} S_n \left (\sum^{n}_{k=1}a^{1+(p-1)/\delta}_kA^{q/\delta}_k \right
   )^{\delta(p(q+1)-q)/(q(p-1))-1}a^{1+(p-1)/\delta}_nA^{q/\delta}_n \nonumber
   \\
   \leq & \delta \left (1+\frac {p}{q(p-1)}\right )\left (1+\frac
   {1}{1/(p-1)+\delta(1+p/(q(p-1))-1} \right) \nonumber \\
   & \cdot \sum^{N}_{n=1} \left (\sum^{n}_{k=1}a^{1+(p-1)/\delta}_kA^{q/\delta}_k \right
   )^{\delta(p(q+1)-q)/(q(p-1))-1}a^{1+(p-1)/\delta}_nA^{q/\delta+1-1/(p-1)-\delta(p(q+1)-q)/(q(p-1))}_n, \nonumber
\end{align}
   where for the first inequality in \eqref{4.2}, we have used the bound
\begin{align*}
   T_n-T_{n-1} \leq \delta \left (1+\frac {p}{q(p-1)}\right ) \left (\sum^{n}_{k=1}a^{1+(p-1)/\delta}_kA^{q/\delta}_k \right
   )^{\delta(p(q+1)-q)/(q(p-1))-1}a^{1+(p-1)/\delta}_nA^{q/\delta}_n,
\end{align*}
    by the mean value theorem and for the second inequality in \eqref{4.2}, we have used the bound
\begin{align*}
   S_n \leq \left (1+\frac
   {1}{1/(p-1)+\delta(1+p/(q(p-1))-1} \right)A^{1-1/(p-1)-\delta(p(q+1)-q)/(q(p-1))}_n,
\end{align*}
    by Lemma \ref{lem2}.

   Now, we set
\begin{align*}
   P=\frac {\delta(p(q+1)-q)}{\delta(p(q+1)-q)-q(p-1)}, \ Q=\frac
   {\delta(p(q+1)-q)}{q(p-1)},
\end{align*}
   so that $P, Q > 1$ and $1/P+1/Q=1$. We then have, by H\"older's inequality,
\begin{align}
\label{4.3}
   &  \sum^{N}_{n=1} \left (\sum^{n}_{k=1}a^{1+(p-1)/\delta}_kA^{q/\delta}_k \right
   )^{\delta(p(q+1)-q)/(q(p-1))-1}a^{1+(p-1)/\delta}_nA^{q/\delta+1-1/(p-1)-\delta(p(q+1)-q)/(q(p-1))}_n
   \\
   = & \sum^{N}_{n=1} \left (\sum^{n}_{k=1}a^{1+(p-1)/\delta}_kA^{q/\delta}_k \right
   )^{Q/P}a^{1/P}_nA^{-(1/(p-1)+Q)/P}_n\cdot a^{1/Q+(p-1)/\delta}_nA^{q/\delta+1-(1/(p-1)+Q)/Q}_n \nonumber \\
   \leq & \left (\sum^{N}_{n=1} a_nA^{-1/(p-1)-\delta(p(q+1)-q)/(q(p-1))}_n
   \left (\sum^{n}_{k=1}a^{1+(p-1)/\delta}_kA^{q/\delta}_k \right
   )^{\delta(p(q+1)-q)/(q(p-1))} \right )^{1/P} \nonumber \\
   & \cdot \left ( \sum^{N}_{n=1}\left (a^p_nA^q_n \right )^{1+1/q} \right
   )^{1/Q}. \nonumber
\end{align}
    It follows from inequalities \eqref{4.2} and \eqref{4.3} that
\begin{align*}
   & \sum^{N}_{n=1} a_nA^{-1/(p-1)}_n
   \left (\sum^{n}_{k=1}\frac {a_k}{A_n}a^{(p-1)/\delta}_kA^{q/\delta}_k \right
   )^{\delta(p(q+1)-q)/(q(p-1))}   \\
   \leq &  \left ( \delta \left (1+\frac {p}{q(p-1)}\right )\left (1+\frac
   {1}{1/(p-1)+\delta(1+p/(q(p-1))-1} \right) \right )^Q \\
   & \cdot \sum^{N}_{n=1}\left (a^p_nA^q_n \right )^{1+1/q} .
\end{align*}
   One sees easily that the above inequality also holds when $\delta=q(p-1)/(p(q+1)-q)$. Combining the above inequality with \eqref{4.1}, we see this
   establishes \eqref{auxbound} with
\begin{align*}
    K_1(p,q) = \min_{\delta} \left ( C \left(p,q, \delta \right )^{1/(1-\theta)}  \right
),
\end{align*}
    where $C(p,q,\delta)$ is defined as in \eqref{8} and the
minimum is taken over the $\delta$'s satisfying \eqref{9} and this
completes the proof of Lemma \ref{basiclem}.
\end{proof}

\begin{lemma}
\label{casep=1}
   Let $p = 1, q>0, r \geq 1$ be fixed. Under the same notions of Theorem \ref{thm1}, we
   have
\begin{align*}
   K(1,q,r) \leq {\displaystyle{ \left\{\begin{array}{ll}
   r^{r-1}K(1+(r-1)/q, q+r-1, 1), &  \ r \geq 2; \\
   r \left(K \left (1+(r-1)/q, q+r-1, 1 \right ) \right )^{r-1}, & \ 1 \leq r \leq 2. \\
\end{array}\right. }}
\end{align*}
\end{lemma}
\begin{proof}
    We may assume $a_n = 0$ whenever $n >
    N$. In this case, on setting
\begin{align*}
    b_n=a_nA^{q}_n, \ c_n=\sum^{N}_{k=n}a^{1+1/q}_k, \
    B_n=\sum^{n}_{k=1}b_k,
\end{align*}
    the left-hand side expression of \eqref{2} becomes
\begin{align*}
   \sum^{N}_{n=1}b_nc^r_n.
\end{align*}
    Note that as $r \geq 1$, we have the
    following bounds:
\begin{align*}
   B_n \leq A^{1+q}_n, \ c^{r}_n -
   c^{r}_{n+1} \leq rc^{r-1}_{n}a^{1+1/q}_n.
\end{align*}
   We then apply partial summation together with the bounds above to
   obtain (with $B_0=c_{N+1}=0$)
\begin{align}
\label{3.6}
 \sum^{N}_{n=1}b_n
   c^{r}_n=\sum^{N}_{n=1}(B_n-B_{n-1})c^{r}_n
   =  \sum^{N}_{n=1}B_n(c^{r}_n-c^{r}_{n+1}) \leq r\sum^{N}_{n=1}a^{1+1/q}_nA^{1+q}_nc^{r-1}_{n}.
\end{align}
    When $r \geq 2$, we apply inequality \eqref{3.5} to see that
\begin{align*}
  \sum^{N}_{n=1}a^{1+1/q}_nA^{1+q}_nc^{r-1}_{n} \leq \left ( \sum^{N}_{n=1}a^{1+(r-1)/q}_nA^{q+r-1}_n
\sum^{N}_{k=n}a_k^{1+1/q}  \right)^{\frac {1}{r-1}} \cdot \left (
\sum^{N}_{n=1}b_nc^{r}_n \right)^{\frac {r-2}{r-1}}.
\end{align*}
    Combining this with inequality \eqref{3.6}, we see that this
    implies that
\begin{align*}
    \sum^{N}_{n=1}b_n
   c^{r}_n \leq r^{r-1}\sum^{N}_{n=1}a^{1+(r-1)/q}_nA^{q+r-1}_n
\sum^{N}_{k=n}a_k^{1+1/q}.
\end{align*}
  The assertion of the lemma for $r \geq 2$ now follows on applying inequality
  \eqref{2} to the right-hand side expression above.

  When $1 \leq r \leq 2$, we apply inequality \eqref{prebound1} in \eqref{3.6} to see
  that
\begin{align*}
    K(1,q,r) \leq  r K \left (1+1/q, 1+q, r-1 \right ) \leq r \left(K \left (1+(r-1)/q,  q+r-1, 1 \right ) \right )^{r-1}.
\end{align*}
  The assertion of the lemma for $1 \leq r \leq 2$ now follows and this
  completes the proof.
\end{proof}

   Now, to establish Theorem \ref{thm2}, it suffices
   to apply Lemma \ref{reductionlem} with the observation that when $q+r-q/p \geq 2$,
   Lemma \ref{casep=1} implies that
\begin{align*}
   K \left (1, \frac {q}{p}, \frac {p(q+r)-q}{p} \right ) \leq \left ( \frac {p(q+r)-q}{p} \right
   )^{\frac {p(q+r-1)-q}{p}}K \left ( p\left (1+\frac {r-1}{q} \right ), q+r-1,
1 \right ),
\end{align*}
   while when $1 \leq q+r-q/p \leq 2$,
   Lemma \ref{casep=1} implies that
\begin{align*}
   K \left (1, \frac {q}{p}, \frac {p(q+r)-q}{p} \right ) \leq \left ( \frac {p(q+r)-q}{p} \right
   ) \left ( K \left ( p\left (1+\frac {r-1}{q} \right ), q+r-1,
1 \right ) \right )^{\frac {p(q+r-1)-q}{p}}.
\end{align*}
   The bound for $K(p,q,1)$ follows from Lemma \ref{basiclem} and this completes the proof of Theorem \ref{thm2}.

\section{Further Discussions}
\label{sec 4} \setcounter{equation}{0}
    We now look at inequality
    \eqref{2} in a different way. For this, we define for any non-negative sequence $(a_n)$ and any integers $N \geq n \geq
    1$,
\begin{align*}
    A_{n,N}=\sum^N_{k=n}a_k, \
    A_{n,\infty}=\sum^{\infty}_{k=n}a_k.
\end{align*}
    We then note that in order to establish inequality
    \eqref{2}, it suffices to show that for any integer $N \geq 1$, we have
\begin{equation*}
   \sum^{N}_{n=1}a^p_nA^q_n\left ( \sum^{N}_{k=n}a_k^{1+p/q} \right )^r \leq
  K(p,q,r)
   \sum^{N}_{n=1}\left (a^p_nA^q_n \right )^{1+r/q}.
\end{equation*}
    Upon a change of variables: $a_n \mapsto a_{N-n+1}$ and recasting, we
    see that the above inequality is equivalent to
\begin{equation*}
   \sum^{N}_{n=1}a^p_nA^q_{n, N} \left ( \sum^{n}_{k=1}a_k^{1+p/q} \right )^r \leq
  K(p,q,r)
   \sum^{N}_{n=1}\left (a^p_nA^q_{n, N} \right )^{1+r/q}.
\end{equation*}
    On letting $N \rightarrow \infty$, we see that inequality \eqref{2} is equivalent to the following inequality:
\begin{equation}
\label{6}
   \sum^{\infty}_{n=1}a^p_nA^q_{n, \infty}\left ( \sum^{n}_{k=1}a_k^{1+p/q} \right )^r \leq
   K(p,q,r)
   \sum^{\infty}_{n=1}\left (a^p_nA^q_{n, \infty} \right )^{1+r/q}.
\end{equation}
    Here $K(p, q, r)$ is also the best possible constant
   such that inequality \eqref{6} holds for any non-negative sequence $(a_n)$.

    We point out that one can give another proof of Theorem \ref{thm2} by
   studying \eqref{6} directly. As the general case $r \geq 1$ can be reduced to the case $r=1$ in a similar way as was done in the proof of
   Theorem \ref{thm2} in Section \ref{sec 3}, one only needs to establish the upper bound for
   $K(p,q,1)$ given in \eqref{7}. For this, one can use an
   approach similar to that taken in Section \ref{sec 3}, in
   replacing Lemma \ref{lem1} and Lemma \ref{lem2} by the
   following lemmas. Due to the similarity, we shall leave the
   details to the reader.
\begin{lemma}
\label{lem0} Let $d \geq c> 1$ and $(\lambda_n)$ be a positive
sequence with $\sum^{\infty}_{k=1}\lambda_k < \infty$. Let $
\Lambda^*_n=\sum^{\infty}_{k=n}\lambda_k$. Then for all
non-negative sequences $(x_n)$,
\begin{equation*}
  \sum^{\infty}_{n=1}\lambda_n(\Lambda^*_n)^{-c}\left ( \sum^{\infty}_{k=n}\lambda_kx_k\right )^d \leq
   \left (\frac {d}{c-1} \right)^d \sum^{\infty}_{n=1}\lambda_n(\Lambda^*_n)^{d-c}x^d_n.
\end{equation*}
    The constant is best possible.
\end{lemma}
   The above lemma is Corollary 6 to Theorem 2 of
   \cite{B1} and only the special case $d=c$ is needed for the
   proof of Theorem \ref{thm2}.

\begin{lemma}
\label{lem1'} Let $p < 0 $. Let integers $M \geq N \geq 1$ be
fixed. For any positive sequence $(a_n)^M_{n=1}$ with $A_{n,
M}=\sum^{M}_{k=n}a_k$, we have
\begin{equation*}
  \sum^{N}_{k=1}a_kA^{p-1}_{k,M} \leq \left (1-\frac {1}{p} \right
  )A^p_{N,M}.
\end{equation*}
\end{lemma}



\begin{thebibliography}{99}
\bibitem{Alzer} H. Alzer, On a problem of Littlewood, {\em   J. Math. Anal. Appl.}, {\bf 199} (1996),  403--408.
\bibitem{Be3} G. Bennett, An inequality suggested by Littlewood, {\em  Proc. Amer. Math. Soc.}, {\bf 100} (1987),  474--476.
\bibitem{B1} G. Bennett, Some elementary inequalities, {\em  Quart. J. Math. Oxford Ser. (2)}, {\bf 38} (1987),  401--425.
\bibitem{CTZ} L.-Z. Cheng, M.-L. Tang and Z.-X. Zhou, On a problem of Littlewood, {\em
 Math. Practice Theory}, {\bf 28} (1998), 314-319  (Chinese).
\bibitem{Copson} E. T. Copson, Note on series of positive terms, {\em  J. London Math. Soc.}, {\bf 3} (1928),  49--51.
\bibitem{G} W.-M. Gong, On a problem of Littlewood, {\em
 J. Yiyang Teachers College}, {\bf 14} (1997), 15-16  (Chinese).
\bibitem{Littlewood} J. E. Littlewood, Some new inequalities and unsolved problems. In: {\em
 Inequalities}, Academic
Press, New York, 1967, 151-162.
\bibitem{Z} D.-Z. Zhang, A further study on the Littlewood problem, {\em  Sichuan Daxue
Xuebao}, {\bf 44} (2007), 956-960 (Chinese).
\end{thebibliography}
\end{document}